\documentclass[12pt,twoside]{amsart}
\usepackage{amsmath}
\usepackage{amssymb}
\usepackage{amscd}
\usepackage{xypic}
\xyoption{all}
\title[Calabi-Yau attractor varieties and degeneration of Hodge structure]{Calabi-Yau attractor varieties and degeneration of Hodge structure}

\author{Mohammad Reza Rahmati}
\thanks{}
\address{Universidad delaSalle Bajio, Leon GTO, Mexico
\hfill\break 
\hfill\break \\
\hfill\break }
\email{mrahmati@cimat.mx}

\newcommand{\comments}[1]{}

\def \A{{\mathcal A}}

\newtheorem{theorem}{Theorem}[section]
\newtheorem{proposition}[theorem]{Proposition}

\newtheorem{definition}[theorem]{Definition}
\newtheorem{remark}[theorem]{Remark}
\newtheorem{example}[theorem]{Example}

\newtheorem{conjecture}[theorem]{Conjecture}



\begin{document}

\begin{abstract}
We study the structure of string theory flux compactification for a general family of elliptic CY 3-folds. We investigate the locus of the attractor points of the flux compactification in type IIB string theory on the boundary components of period domains. Specifically we give equations describing this locus through the asymptotic of nilpotent orbits on period domains. our approach is a mixture of techniques of asymptotic Hodge theory and the numerical period vectors used in physics. \end{abstract}

\maketitle

\vspace{0.5cm}

\section{Introduction}

\vspace{0.2cm}

A Calabi-Yau (CY) manifold is a compact K\"ahler manifold with trivial canonical class. An example is a smooth projective variety defined by a homogeneous polynomial of degree $(n+2)$ in complex projective space $\mathbb{P}^{n+1}$. For $n=1$ this gives an elliptic curve, and for $n=2$ it yields a K3 surface. For $n=3$ the classification is open. Calabi-Yau varieties appear in the context of mirror symmetry. That is they usually appear as an isomorphic pair and also isomorphic families of varieties. A famous example of the mirror families is the one parameter family 
\begin{equation}
f= z_0^{n+1}+z_1^{n+1}+...+z_n^{n+1}+(n+1)\psi z_0z_1...z_n=0
\end{equation}
due to Dwork, \cite{Dw}. For $n=4$ is the aforementioned quintic family of Candelas et al. The mirror family is defined by 
\begin{equation}
y_1y_2...y_n(y_1+y_2+...+y_n+1)+\frac{(-1)^{n+1}\phi}{(n+1)^{n+1}}=0    
\end{equation}
which defines a family over $\mathbb{P}^1$. The family maybe studied iteratively starting from $n=0$, where both of the equations define a family of quadrics in $\mathbb{P}^1$, and hence are isomorphic fibre-wise, \cite{CGP, COES, CO, CK, COFKM, Ya1, Ya2, Ya3, Ya4, Ya5, Yu2, Yu3}. We consider the projective family of CY 3-folds 
\begin{equation}
f:\mathfrak{X} \longrightarrow \mathbb{P}^1 \setminus 0, 1, \infty
\end{equation}
Each fiber $X_s=f^{-1}(s)$ is a compact complex 3-fold with trivial canonical bundle $\omega_{X_s}=\bigwedge^3 T^*X_s$. The family of middle cohomologies 
\begin{equation}
\mathcal{H}=R^3f_*\mathbb{C} \longrightarrow \mathbb{P}^1 \setminus 0, 1, \infty
\end{equation}
carries a flat connection called the Gauss-Manin connection. Such a family maybe constructed iteratively from a family of 1-lower dimension by quadratic twist, see Section 1 below. If $\xi \in \Gamma([\mathbb{P}^1 \setminus 0 , 1, \infty]; \mathcal{H})$ then the local analytic function 
\begin{equation}
\Pi_{\xi}:s \longmapsto \int_{\gamma_s} \xi
\end{equation}
is called a period integral. The Gauss-Manin connection $\nabla$ reduces to a linear differential equation 
\begin{equation}
{L}_n(s) \Pi_{\xi}(s)=0
\end{equation}
of rank $n$ namely Picard-Fuchs equation. 
There is a meromorphic $(n+1)$-form on $\mathfrak{X}$ 
\begin{equation}
\Omega= \sum_j\frac{(-1)^j\psi}{f}(z_idz_0 \wedge ... \wedge \widehat{dz_i} \wedge ... \wedge dz_{n+1})  
\end{equation}
such that $\Omega$ restricts to a holomorphic $n$-form $\Omega_{\psi}=\text{Res}_{X_s}(\Theta)$ on the smooth fibres. Then the periods of the Hodge structure of the fibration are integrals of the form $\int_{\Sigma}\Omega_{\psi}$ for $\sigma \in H_n(X_s, \mathbb{Z})$. The Hodge filtration on $H^3(X_{\psi}, \mathbb{C})$ is given by 
\begin{equation}
\begin{aligned}
F^3&=\langle \Omega_{\psi} \rangle\\
F^2&=\langle \Omega_{\psi} , \Omega_{\psi}' \rangle\\
F^1&=\langle \Omega_{\psi} , \Omega_{\psi}', \Omega_{\psi}'' \rangle\\
F^0&=\langle \Omega_{\psi} , \Omega_{\psi}', \Omega_{\psi}'', \Omega_{\psi}''' \rangle
\end{aligned}
\end{equation} 
There is another method to compute the periods by solving the Picard-Fuchs equation
\begin{equation}
\left ( \Theta^{n+1}-\psi \prod_{k+1}^{n+1} (\Theta+\frac{k}{n+2}) \right )\Omega_{\psi}=0, \qquad \Theta=\psi \frac{d}{d\psi} 
\end{equation}
By a choice of basis $(\alpha_i)$ for $H_n(X_{\psi}, \mathbb{Z})$, we can form the period vectors $\Pi_{\Omega}=[\Pi_i=\int_{\alpha_i}\Omega_{\psi}]$. It is known that the forms 
\begin{equation}
\Omega_{\psi}, \Theta \Omega_{\psi},..., \Theta^n \Omega_{\psi}
\end{equation}
are linearly independent in $H^n(X_{\psi},\mathbb{C})$. It follows that we may form the vectors 
\begin{equation}
\Pi_{\Omega}, \Theta \Pi_{\Omega},..., \Theta^n \Pi_{\Omega}
\end{equation}
in a matrix to obtain a partial period matrix, \cite{FC, GGK2, GGK, KPR, KP, KP1, KNU, Ya1, Ya2, Ya3, Ya4, Ya5}. In fact this construction applies to other smooth families. The first example is the family of elliptic curves 
\begin{equation} 
y^2=4x^3-g_2(t)x-g_3(t), \qquad t \in \mathbb{P}^1
\end{equation}  
where $g_2$ and $g_3$ are polynomials of degrees at most $4$ and $6$ respectively. Set $\Delta=g_2^3-27g_3^2, \  j=g_2^3/\Delta$ where $g_2, \ g_3, \ \Delta, \ j$ are Weierestrass coefficients, discriminant and $J$-function. The Picard-Fuchs equation for the family is given by 
\begin{equation}
\frac{d}{dt}\begin{bmatrix} \omega \\ \eta \end{bmatrix} =\begin{pmatrix} 
\frac{-1}{12}\frac{d \log \delta}{dt} & \frac{3\delta}{2 \Delta}\\
\frac{-g_2\delta}{8 \Delta} & \frac{1}{12}\frac{d \log \delta}{dt}
\end{pmatrix}\begin{bmatrix} \omega \\ \eta \end{bmatrix}, \qquad \delta=3g_3g_2'-2g_2g_3'
\end{equation}
where $\omega=\int_{\gamma}\frac{dx}{y}, \ \eta =\int_{\gamma}\frac{xdx}{y}$ and $\gamma$ being a $1$-cycle, \cite{DM}. The local system $\mathcal{H}=R^1\pi_*\mathbb{C}$ has a two step Hodge filtration 
\begin{equation} 
F^0=H^1(X_t, \mathbb{C}) \supset F^1
\end{equation} 
The Gauss-Manin action on a multivalued section of the canonical extension of $\mathcal{H}$ is given as 
\begin{equation}
\nabla:\ \alpha^*+\left (n_i\frac{\log s}{2\pi \sqrt{-1}}+\psi(s) \right )\beta^* \longmapsto  \left (n_i\frac{d s}{2\pi \sqrt{-1}s}+\psi'(s)ds \right )\beta^*
\end{equation}
where $\alpha$ and $\beta$ are a symplectic basis of $H_1(X_s, \mathbb{C})$, \cite{GGK}. The next example is a family of $K3$ surfaces defined by 
\begin{equation} 
y^2=4x^3-G_2(t,s)x-G_3(t,s)
\end{equation} 
where $G_2, \ G_3$ are polynomials of degree at most $8, \ 12$ in the affine coordinate $s$ and such that they are also polynomials in $t$. The periods are calculated via the integral $\int_{\gamma} ds \wedge \frac{dx}{y}$ and the local systems $\mathcal{H}^2=R^2\pi_*^2\mathbb{C}$ has a weight $2$ Hodge filtration 
\begin{equation} 
F^0=H^2(X_{(t,s)},\mathbb{C}) \supset F^1 \supset F^2
\end{equation} 
An illustration for the second family is 
\begin{equation}
\begin{CD}
E_t @>>> \mathcal{E} @>>>\mathbb{H} \\
@VVV  @VVV  @VV{j(\tau)}V\\
t @>>> \mathbb{P}^1  @>>{j(t)}>\mathbb{P}_{j-line}^1\\
\end{CD} 
\end{equation}
where $j$ is the $j$-function of the fibres, cf. \cite{GGK}. One may proceed inductively to 3-dimensional fibrations over surfaces, 
\begin{equation} 
y^2=4x^3-G_2(s,t,u)x-G_3(s,t,u)
\end{equation} 
having singular fibres $t=0, 1, \infty$, where $s$ and $u$ are affine variables on a surface. We have $X_t \to \mathbb{P}^1$. The cohomologies of the smooth fibres build up a VHS of type $(1,1,1,1)$, whose periods are given by certain hypergeometric functions, \cite{DM}, see also \cite{B, BP, CK, CKS, CK, COFKM, Car, Di, Dol, DKSSVW, Dw, DM, DM1, DHT, FC, GGK2, GGK, Har, Hyb, KP1, Sch, St}.

A Calabi-Yau $n$-fold $X$ is called an attractor if there exists an integral cohomology class $\gamma \in H^n(X, \mathbb{Z})$ such that $\gamma \perp H^{n-1,1}, \ \text{under the symplectic inner product on} \ H^n(X, \mathbb{C})$.
It is supposed that the attractors furnish special points analogous to complex multiplication points on Shimura varieties. Attractor CY varieties appear in the context of supersymmetric flux compactification in type IIB mirror symmetry, \cite{AMo, BGH, BCV, Be, BKSW, CGV, CKS, COFKM, DKMM, DKMM0, GLV, Gr, GBBS, GVW}. 
Another way to reach to this concept is by the string theory superpotential defined by
\begin{equation} 
Z_{\gamma}=e^{\sqrt{-1} \left \langle  \Omega, \overline{\Omega} \right \rangle /2}\int_{\gamma}\Omega
\end{equation} 
where $\Omega \in H^3(X, \mathbb{Z})$. By choosing a symplectic basis $(A_i,B_i)$ for $H_3(X, \mathbb{Z})$ we may write $\gamma=\sum_i a_iA_i-b_iB_i$. Set, $Q=(b_j, a_j)^t$. Then we have 
\begin{equation}
Z_{\gamma}=\frac{Q^t \Sigma \Pi}{(-i\Pi^{\dagger}\Sigma \Pi)^{1/2}}   
\end{equation}
where $\Sigma$ is the intersection form in the symplectic basis. The attractor points minimize the function $|Z_{\gamma}|$ \cite{M, KNY}, see also \cite{HK, JK, Ka, KLL, LY, LMMV2, LMMV, LLW5, LLW4, LLW3, S}. The paper applies the degeneration techniques of variation of Hodge structure (VHS) to study the locus of attractor CY varieties. In Section 4 we review known techniques for partial toridal compactification of the period domain $D$ of the Hodge structure $(V_{\mathbb{Q}}, Q)$ [$Q$ is a nondegenerate biliniear form on $V$]. In the specific compactification the boundary components are associated to the cones of nilpotent transformations in $\mathfrak{g}=Lie(Aut(V, Q)_{\mathbb{C}})$. Given $N \in \mathfrak{g}$, the set 
\begin{equation}
\check{B}(N)=\{\ F \in \check{D}\ |\ Ad(e^{\tau N})F \in D, \text{Im}(\tau) \gg 0\ \}
\end{equation} 
denotes all arising possible LMHS associated to $N$ [$\check{D}$ is the compact dual of $D$]. The quotient 
\begin{equation} 
B(N)=e^{\mathbb{C}N} \setminus \check{B}(N)
\end{equation}
consists of all the LMHS upto reparamentization. We define the boundary component associated to $N$ by $\overline{B(N)}=\Gamma \setminus B(N)$, \cite{K, KP, KPR, KP1}. We propose to find equations that provide constrains for the attractive locus on the boundaries of $M$. That is we try to check out if 
\begin{equation}
\lim_{\text{Im}(z) \to \infty}\exp(-\sqrt{-1}z.N)\Pi
\end{equation}
will belong to $H^{3,0} \oplus H^{0,3}$. Note that the boundary component depends to $N$ and therefore we have to consider several constrains for different boundary components. We explain this by a specific example in Section 5. Some more explanation on the asymptotic behavior of the CY periods is given there. 

We shall consider the complex structure moduli of the CY variety $X$ as a symplectic manifold with the action of the  lie group $G_{\mathbb{R}}=Aut(\Sigma, \mathbb{R})$ of its symmetries. Assume  
\begin{equation} 
\mu:M \to \mathfrak{g}
\end{equation} 
is a moment map for the action, see \cite{VRS}. We will use (21) to define a moment map on the complex structure moduli space $M$ of the CY variety $X$. We prove that the attractor points on the moduli of complex structure $M$, are the critical points of the negative gradient flows of $-\nabla |Z_{\gamma}|$. This gives an alternative interpretation for the attractor locus on $M$. The conclusion is that using known results of GIT for reductive groups the attractors are characterized as the limit points on the integral curves of the negative gradient flow. In other words the limit at infinity on the integral curves also minimizes the absolute value $|Z_{\gamma}|$.

In Section 7 we give an overview of the connections of attractors to some concepts in arithmetic and modularity conjecture for CY 3 folds. In a recent work by Candelas dela Ossa et. al, some examples of modular CY varieties has been studied by checking evidences through the attractor formalism. In fact the attractor formalism suggests in a weight 3 Hodge structure of type $(1,1,1,1)$ to check out if the decomposition of the Hodge structure on $H^3(X, \mathbb{Q})$
as $HS_1 \oplus HS_2$ may be lifted to the motive of the variety $X$. This method is getting done by checking if the corresponding $p$-factor may factorize for some primes of good reduction. In fact in the example of Candelas et al. this unquely identifies certain modular forms whose $L$-functions is equal to the $L$-function of the Galois representation on the \'etale cohomology of $X$, \cite{CGP, COES, CO, CK, COFKM, Car, Cha, Del, Di, Dol, DKSSVW, Dw, DHT, DK, GGK2, GHJ, Gr, GVW, Har, Hyb, HK, KY, KNY, Ka, M, Mi, Mor, Si1, Si2, Yu2, Yu3}
.   
\vspace{0.3cm}

\noindent
\textbf{Explanation on the text and references:} Section 1 is the introduction containing also some basics of Hodge theory. Section two is a brief from \cite{DM} and explains hypergeometric fibration of CY varieties. Section 3 is a brief on attractors and flux compactification mainly from \cite{M}. Section 4 explains a partial toroidal compactification of period domains of Hodge structure, with main references \cite{K, KP,GGK}. In Section 5 we study attractors and the asymptotic Hodge theory. We give a method to compute attractors by the asymptotic boundary in the moduli of HS or period domains. In Section 6 we present a result on the interpretation of attractive locus on the complex structure moduli as critical points of a moment map. In Section 7 we present an overview of connection between attractors and modularity conjecture for CY 3-folds and some relations with class field theory.  
\vspace{0.3cm}

\section{Elliptically fibered family of CY 3-folds}
\vspace{0.2cm}

This section is a brief from \cite{DM} to introduce hypergeometric families of CY manifolds elliptically fibred over varieties of smaller dimensions. The reader can consult with the reference for more details. The elliptically fibred family of CY 3-folds can be analysed iteratively by starting from a family of two points 
\begin{equation} 
y^2=1-s
\end{equation} 
It is isomorphic to a quadric pencil. By choosing a suitable branch cut we obtain a holomorphic section $y^{-1}={}_1F_0\left (\frac{1}{2}|s \right )=(1-s)^{-1/2}$, where ${}_1F_0$ is the hypergeometric function. It is annihilated by the operator 
\begin{equation} 
L_1(s)=\Theta-s(\Theta+1/2), \qquad \Theta=s\frac{d}{ds}
\end{equation} 
with monodromy equal to $1, -1, -1$ over the points $0, 1, \infty$. 

Applying the twist replacement 
\begin{equation} 
y^2 \longmapsto -y^2/x(x-s)
\end{equation} 
in the equation of the former family we obtain the family 
\begin{equation} 
y^2=x(x-1(x-s), \qquad s \in (\mathbb{P}^1 \setminus \{0, 1, \infty\})
\end{equation} 
where the middle cohomologies of the fibres $\mathcal{H}_s=H^1(X_s, \mathbb{C})$ provides a local system of polarized Hodge structure of weight 1 and of type $(1,1)$, where $F_{0,s}$ is the whole $H^1(X_s, \mathbb{C})$ and $F_1=H^{1,0}(X_s)$. the space of harmonic 1-forms is generated by the 1-form $dx/y$. The periods are defined via the integral 
\begin{equation} 
\int {dx}/{y}=\int_0^s{dx}/{\sqrt{x(x-1)(x-s)}}=(i\pi) 
{}_2F_1\left (  _{\ \ 1}^{\frac{1}{2},\frac{1}{2}} \big | \ s \right )={}_1F_0\left (\frac{1}{2}|s \right ) * {}_1F_0\left (\frac{1}{2}|s \right )
\end{equation} 
where the $*$ is the Hadamard product of power series 
\begin{equation} 
(\sum a_n s^n )* (\sum b_n s^n) =\sum a_n b_n s^n
\end{equation} 
The Picard-Fuchs operator becomes 
\begin{equation} 
L_2(s)=\Theta^2-s(\Theta+1/2)^2.
\end{equation} 
We may repeat this process by applying the same twist to get a family of K3 surfaces 
\begin{equation} 
y^2=x(x-1)(x-t)t(t-s)
\end{equation}  
over $\mathbb{P}^1$ where the singular fibres lie over $0, 1, t, \infty$. Again the middle cohomologies of the smooth fibres defines a Hodge structure of weight 2. we can stress the periods in two ways. The first is through the twist we applied above. For suitable varying 1-cycles $\gamma_1, \gamma_2$ that form a basis for the cohomology of the elliptic fibre the periods of the middle cohomology are calculated via the integral $\int_s^tdt\begin{pmatrix}\int_{\gamma_1}dx/y \\ \int_{\gamma_2}dx/y \end{pmatrix}$. For particular choice of branches the periods are related to 
\begin{equation} 
\int_0^s\left (dt/\sqrt{t(t-s)}\right ) 
{}_2F_1\left (  _{\ \ 1}^{\frac{1}{2},\frac{1}{2}} \big | \ s \right )=(i\pi) 
{}_3F_2\left (  _{\ \ 1, 1}^{\frac{1}{2},\frac{1}{2}, \frac{1}{2}} \big | \ s \right )={}_1F_0\left (\frac{1}{2}|s \right ) * (i\pi){}_2F_1\left (  _{\ \ 1}^{\frac{1}{2},\frac{1}{2}} \big | \ s \right )
\end{equation}
The second, the original period is given via integrating the 2-form $dt \wedge dx/y$ over a two cycles $\Gamma$ where it results 
\begin{equation} 
\int_{\Gamma}dt \wedge dx/y=(i \pi)^2  
{}_3F_2\left (  _{\ \ 1, 1}^{\frac{1}{2},\frac{1}{2}, \frac{1}{2}} \big | \ s \right )
\end{equation} 
The Picard-Fuchs operator is given 
\begin{equation} 
L_3(s)=\Theta^3-s(\Theta+1/2)^3.
\end{equation}
Applying the next twist we get the family of CY 3-folds 
\begin{equation} 
y^2=x(x-1)(x-t)t(t-w)w(w-s)
\end{equation} 
We denote $f:\mathfrak{X} \to \mathbb{P}^1$ to be the family parametrized with $s$, where the family of middle cohomologies 
\begin{equation} 
R^3f_*\mathbb{C} \to (\mathbb{P}^1 \setminus 0, 1, \infty)
\end{equation} 
has a Hodge structure of weight 3 and of type $(1, 1, 1, 1, 1)$. By choosing branches the periods are related to 
\begin{equation} 
\int_0^s\left (dt/\sqrt{w(w-s)}\right ) 
{}_3F_2\left (  _{\ \ 1, 1}^{\frac{1}{2},\frac{1}{2}, \frac{1}{2}} \big | \ s \right )=(i\pi) 
{}_3F_2\left (  _{\ \ 1, 1, 1}^{\frac{1}{2},\frac{1}{2}, \frac{1}{2}, \frac{1}{2}} \big | \ s \right )={}_1F_0\left (\frac{1}{2}|s \right ) * (i\pi) 
{}_2F_1\left (  _{\ \ 1}^{\frac{1}{2},\frac{1}{2}, \frac{1}{2}} \big | \ s \right )
\end{equation} 
The second way to present the periods, is by taking the original period given via integrating the 2-form $dw \wedge dt \wedge dx/y$ over a two cycles $\Gamma$ where it results 
\begin{equation} 
\int_{\Gamma}dw \wedge dt \wedge dx/y=(i \pi)^2  
{}_3F_2\left (  _{\ \ 1, 1}^{\frac{1}{2},\frac{1}{2}, \frac{1}{2}} \big | \ s \right )
\end{equation} 
The Picard-Fuchs operator is given 
\begin{equation} 
L_3(s)=\Theta^4-s(\Theta+1/2)^4
\end{equation}
This process continues similarly to higher dimensional fibrations. 

The aforementioned construction maybe done using the Weierstrass normal forms in the equation of fibration, and in that case one obtains hypergeometric functions with slightly different weights. However the same relation exists between the Picard-Fuchs equations and the periods. In general the hypergeometric function 
\begin{equation} 
F(s)={}_nF_{n-1}\left (  _{\ \ 1, 1, ...,  1}^{\rho_1,\rho_2, ..., \rho_n} \big | \ s \right )
\end{equation} 
satisfies the  differential equation
\begin{equation} 
\left [ \Theta^n-s(\Theta+\rho_1)...(\Theta+\rho_n) \right ]F(s)=0
\end{equation}
The construction also applies to the mirror pair. That is on the mirror family we also have an iterative formulas involving the periods and the differential operators. An example is the mirror for the Dwork family stated in equation (2) in the introduction. In this case according to [\cite{DM} theorem 10.2] we have 
\begin{equation} 
\Pi_{n-1}(s)=(2i\pi){}_nF_{n-1}\left (  _{\ \ \frac{1}{n},\frac{2}{n}, ..., \frac{n-1}{n}}^{\frac{1}{n+1},\frac{2}{n+1}, ..., \frac{n}{n+1}} \big | \ s \right )*\Pi_{n-2}(s).
\end{equation} 
In all the examples constructed iteratively we have singular fibres over $0, 1, \infty$. We can normalize the local solutions by setting 
\begin{equation} 
f_{\alpha}(t)=t^{\alpha}{}_nF_{n-1}^{\alpha}\left (  _{\ \ 1, 1, ...,  1}^{\rho_1,\rho_2, ..., \rho_n} \big | \ t \right )=\sum_j\frac{(\rho_1+\alpha)_j...(\rho_n+\alpha)_j}{(1+\alpha)_j^n}t^{j+\alpha}=\sum_{l=1}^{n-1}(2i \pi \alpha)^lf_l(t)
\end{equation} 
where $f_l=\frac{1}{(2i\pi)^l l!} \frac{\partial^l}{\partial \alpha^l}|_{\alpha=0}f_{\alpha}(t)$ and $(\rho)_j=(-1)^j\frac{\Gamma(1-\rho)}{\Gamma(1-k-\rho)}$. The functions $(f_{n-1},...,f_0)$ form a basis of the local solutions near $t=0$ and the monodromy at $t=0$ in this basis is given by 
\begin{equation}
M_0=\begin{pmatrix} 1 & 1 & 1/2 & ...& 1/(n-2)!\\ 0 & 1 & 1 & ... & 1/(n-3)!\\ ... & ... & & ...& \\ 0 & ... & ... & 0 & 1 \end{pmatrix}     
\end{equation}
Similarly by introducing the functions 
\begin{equation} 
F_k({\beta},t)=B_kt^{-\beta}{}_nF_{n-1}^{\alpha}\left (  _{\ \ 1+\beta_k-\beta_1,\ ... \hat{1} ...,\  1+\beta_k-\beta_n}^{\ \ \ \ \ \ \ \ \ \  \rho_k,\ \rho_k,\ ...,\ \rho_k} \big | \ \frac{1}{t} \right )
\end{equation} 
for specific numbers $B_k$, we can write down the the basis of solutions at $t=\infty$. According to the analysis in \cite{DM} the monodromy at $\infty$ with respect to the basis $(F_n,...,F_1)$ is given by 
\begin{equation} 
M_{\infty}=[e^{-2i\pi\beta_n}, e^{-2i\pi\beta_{n-1}}, ..., e^{-2i\pi\beta_1}]
\end{equation} 
The above two monodromy operators are related by the transition matrix between the analytic continuations of the two basis, denoted by $P$. We have $M_{\infty}=P.M_0.P^{-1}$, and the monodromy at the other singularity is given by 
\begin{equation} 
M=M_{\infty}M_0^{-1}.
\end{equation}
cf. \cite{DM} loc. cit. see also \cite{B, BP, CK, CKS, CK, COFKM, Car, Di, Dol, DKSSVW, Dw, DM, DM1, DHT, FC, GGK2, GGK, Har, Hyb, KP1, Sch, St}. 
\vspace{0.3cm}

\section{Attractor points in moduli of complex structures}

\vspace{0.2cm}

The material in this section is well known. We explain two kinds of attractors in dimension 3 with origins from string theory. Lets begin from the definition of the attractor point in the moduli of $X$. 
\begin{definition} 
A Calabi-Yau $n$-fold $X$ is called an attractor if there exists an integral cohomology class $\gamma \in H^n(X, \mathbb{Z})$ such that 
\begin{equation} 
\gamma \perp H^{n-1,1} \qquad \text{under the symplectic inner product on} \ H^n(X, \mathbb{C})
\end{equation} 
where we have considered the Hodge decomposition 
\begin{equation}
H^n(X, \mathbb{C})=H^{n,0}(X) \oplus H^{n-1,1}(X) \oplus ... \oplus H^{0,n}(X)   
\end{equation}
\end{definition}
Let $h^{i,j}=\dim H^{i,j}(X)$ be the Hodge numbers. The definition puts $h^{n-1,1}$ conditions on the coordinates in the moduli of Hodge structure (period domain) of $X$. The points in the moduli space corresponding to attractor varieties are called attractor points. It is a conjecture that an attractor variety is defined over $\overline{\mathbb{Q}}$. The concept is reasonably of interest in Physics when $n=3$. In this case the above conjecture is due to Moore and was supposed that the attractors furnish special points analogous to complex multiplication points on Shimura varieties, \cite{AMo, BP, BGH, BCV, Be, BKSW, CGV, CKS, Ce2, Ce, COFKM, DKMM, DKMM0, GLV, Gr, GBBS, GVW}. 

Let $X$ be a Calabi-Yau 3-fold and $M$ be the moduli of complex structures on $X$. It is known that $\dim(M)=h^{2,1}$. Choose a generator $\Omega \in H^{3,0}(X)$ and let $\gamma \in H_3(X, \mathbb{Z})$ its poincar\'e dual. Also let $A_j, \ B_j$ be a symplectic basis for $H_3(X)$, and  $\alpha_j, \ \beta_j$ the dual basis. Consider the flat coordinates $x_j=\int_{\alpha_j} \Omega, y_j=\int_{\beta_j} \Omega$, and $Y=[a_1, \dots , a_n, b_1, \dots , b_n]$. Then the attractor condition may be written as 
\begin{equation}
\bar{c}Y-c\overline{Y}=\sqrt{-1}\Omega \qquad c =\text{const}, 
\end{equation} 
The attractor points can be equally defined by the fixed points of the gradient flow of the function 
\begin{equation} 
Z_{\gamma}=\frac{|\langle \gamma , \Omega \rangle |}{\langle \Omega, \overline{\Omega} \rangle}, \qquad \gamma \in H_3(X, \mathbb{Z})
\end{equation}
on the moduli space $M$. Then the attractor corresponding to the cohomology class $\gamma$ are the fixed points of the gradient flow, \cite{M, CO, COES, CK, HK, JK, Ka, KLL, LY, LMMV2, LMMV, LLW5, LLW4, LLW3, S}. 

In case the space time is equipped with the metric 
\begin{equation} 
ds^2=-e^{2U}dt^2+e^{-U}dx^2
\end{equation} 
The charge of the black hole is given by $\gamma \in H_3(X,\mathbb{Z})$ the Poincar\'e dual to the form $\Omega$. The complex structure of $X$ is getting varied by the radius $r$ described by the differential equation 
\begin{equation}
    dU/d\rho=-e^U|Z_{\gamma}|, \qquad d\phi/d\rho =\sum_{j=1}^{h^{2,1}}-2e^Ug_{j \bar{j}}\partial_{\bar{j}}|Z_{\gamma}|
\end{equation}
where $\rho=1/r$ and the index $j$ runs through a specific basis $(\eta_j)$ of $H^{2,1}(X)$. The basis can be written specifically by considering local holomorphic coordinates on the moduli $M$ of complex structures on $X$. One has $T_s^{1,0}M=H^{2,1}(X_s), \ (s \in M)$ on the complex structure moduli, and the holomorphic coordinates maybe denoted by $z_i, \ 1 \leq i \leq h^{2,1}$. then the basis can be chosen to be 
\begin{equation}
    \eta_j=e^{K/2}\pi^{2,1}(\partial_{z_i} \Omega)=e^{K/2}\left (\partial_{z_i} \Omega -\frac{\langle \partial_{z_i} \Omega , \overline{\Omega} \rangle }{ \langle  \Omega, \overline{\Omega}  \rangle}\Omega \right )
\end{equation}
where $\pi^{2,1}$ is the projection on the $H^{2,1}$ factor in the Hodge decomposition and the differentiation is considered as part of the period  coordinates involving $M$. The function 
\begin{equation} 
K=\log \left (\sqrt{-1} \left \langle  \Omega, \overline{\Omega} \right \rangle \right )
\end{equation} 
is the Kahler potential of WPZ metric on $M$. Then $g_{i\bar{j}}=-\sqrt{-1}\langle \eta_i, \overline{\eta_j} \rangle $. 
\begin{equation} 
Z_{\gamma}=e^{\sqrt{-1} \left \langle  \Omega, \overline{\Omega} \right \rangle /2}\int_{\gamma}\Omega
\end{equation} 
is a function. These equations can be understood as a gradient flow of the function $|Z_{\gamma}|$ with respect to the metric. By choosing a symplectic basis $(A_i,B_i)$ for $H_3(X, \mathbb{Z})$ we may write $\gamma=\sum_i a_iA_i-b_iB_i$. Let $(\alpha_j, \beta_j)$ be the dual basis and $\Omega=\sum b_i\alpha_i-a_i\beta_i, Q=(b_j, a_j)^T$. Then we have 
\begin{equation}
Z_{\gamma}=\frac{Q^t \Sigma \Pi}{(-i\Pi^{\dagger}\Sigma \Pi)^{1/2}}   
\end{equation}
where $\Sigma$ is the intersection form in the symplectic basis. The attractor points minimize the function $|Z_{\gamma}|$, \cite{M, KNY}.  

The attractor definition comes from type IIB string theory on the variety $\mathbb{R}^4 \times X$ where $\mathbb{R}^4$ is the Lorenzian 4-manifold, i.e. $\mathbb{R}^{1,3}$ and $X$ is a CY variety of complex dimension 3 or real dimension 6. On this space we shall have a theory of supergravity which arises from a self dual 5-form in 10-real dimension called type IIB supergravity. This form is presented by the form 
\begin{equation} 
F \in \Omega^2(\mathbb{R}^4) \otimes H^3(X,\mathbb{R})
\end{equation} 
called the electro-magnetic field. The self duality means 
\begin{equation} 
F=* \ F, \qquad \text{in 10-dimension}, 
\end{equation} 
where $*$ is the Hodge star operator. We deal with a complex symplectic structure on $H^3(X, \mathbb{C})$ defined by a linear transformation 
\begin{equation} 
J:H^3(X, \mathbb{C}) \to H^3(X, \mathbb{C}), \qquad J^2 =-1, \ \langle J x,J y \rangle =\langle x,y \rangle
\end{equation} 
We can define an operator $*_T=*_4 \otimes J$. It satisfies $*_T^2=+1$, and $F=-*_TF$.  Flux vacua in type IIB theory has contribution from two fields defined by the 3-forms $F_{(3)}$ and $H_{(3)}$ and is formulated by the superpotential $W$ which is written in terms of the complex axion-dilaton $\tau=C_0+\sqrt{-1}e^{-\phi}$ and the field $G_{(3)}=F_{(3)}-\tau H_{(3)}$ as the integral
\begin{equation}
W=\int_XG_{(3)} \wedge \Omega
\end{equation}
where $\Omega$ is a holomorphic 3-form. The parameter $\tau$ lies in an upper half plane and the vacuum constraints can be written as $D_{\tau}W=D_iW=0$ where $D_{\tau}= \partial_{\tau}+\partial_{\tau}K, D_j= \partial_{j}+\partial_{j}K$, \cite{S}.

\vspace{0.3cm}

\section{Boundary components on moduli of  Hodge structure of CY 3-folds}

\vspace{0.2cm}

In this section we review a partial toroidal compactification of period domain of Hodge structure. The references are \cite{GGK, KP, K}. A polarized Hodge structure of weight $n$ on a $\mathbb{Q}$ vector space $V$ together with a $(-1)^n$-symmetric non-degenerate $\mathbb{Q}$-bilinear form $Q : V \times V \to \mathbb{Q}$, is given by a representation  
\begin{equation} 
\phi: S^1 \to Aut(V_{\mathbb{R}},Q)
\end{equation} 
defined over $\mathbb{R}$ such that if we denote the $t^p\bar{t}^q$-eigenspace of $\phi(t)$ by $V^{p,q}$, then $p+q=n$ for all non-zero $V^{p,q}$ [$t$ denotes the coordinate on $S^1$]. We denote $\textbf{C}:=\phi(i)$. The adjoint action of the group $G_{\mathbb{R}}=Aut(V_{\mathbb{R}},Q)$ on $\phi$ defines the period domain of polarized Hodge structures of weight $n$ on $V$, denoted by $D$. Denote the centralizer of $\phi$ by $M=Z_{\phi}(G_{\mathbb{R}})$, then $D=G_{\mathbb{R}}/M$. It embeds in $\check{D}=G_{\mathbb{C}}/P$, where $P$ is the centralizer of $\phi$ ($P$ is a parabolic subgroup of $G$), and inherits a complex structure from this embedding. 

A basic example of this is the VHS $\mathcal{V}=\mathcal{V}^{1,0} \oplus \mathcal{V}^{0,1}$ of weight $1$, obtained from the middle cohomology of a fibration of curves of genus $g$. In this case 
\begin{equation} 
D=\textbf{H}_g=\{Z \in M_{g \times g} : Z=^tZ, \ Im (Z) >0\}
\end{equation} 
is the Siegel generalized upper half space. A variation of Hodge structure $\mathcal{V}$ on a quasi-projective variety $S$ can be given by its period map 
\begin{equation} 
\Pi:S \to \Gamma_{\mathbb{Z}} \backslash D
\end{equation}  
where $\Gamma_{\mathbb{Z}}$ is a discrete subgroup of $G$ called monodromy group. Let $h=(h^{p,q})_{p+q=n}$ be the $h$-vector of Hodge numbers. 
\begin{itemize} 
\item When $weight=n=2k+1$ is odd we have 
\begin{equation} 
D_h=SP_n(\mathbb{R})/\prod_{l \leq k}U(h^{l,n-l})
\end{equation} 
\item When $n=2k$ set $h_{\text{odd}}=\sum_{l \ \text{odd}}h^{l,n-l}, \ h_{\text{even}}=\sum_{l \ \text{even}}h^{l,n-l}$. Then 
\begin{equation} 
D_h=SO(h_{\text{odd}}, h_{\text{even}})/ \left (SO(h^{k,k} \times \prod_{l < k}U(h^{l,n-l}) \right ).
\end{equation} 
\end{itemize}

Associated to each nilpotent transformation $N \in \mathfrak{g}:=Lie(G_{\mathbb{C}})$ one defines a limit mixed Hodge structure (LMHS), where 
\begin{equation} 
F_{\lim}=\lim_{\text{Im}(z) \to -\infty}e^{-\sqrt{-1}zN}F_z
\end{equation} 
and the weight filtration is the unique one associated to $N$ according to the Jacobson-Morozov theorem. The LMHS depends upon the choice of the coordinate $z$. The re-parametrization has the effect 
\begin{equation}
F_{\lim} \mapsto e^{\alpha N}F_{\lim}
\end{equation} 
We define the nilpotent orbit as the set of flags $e^{\mathbb{C}N}F_{\lim}$, i.e LMHS modulo reparametrization. In order to compactify the image of period map we partially compactify $\Gamma \setminus D$ by spaces denoted by $\overline{B_{\sigma}}$ where $\sigma$ is a cone of nilpotent transformations in $\mathfrak{g}$. Let $T_1,...,T_n$ be generators of the monodromy group $\Gamma$ and $N_j=\log T_j$. Let also 
\begin{equation} 
\sigma_{\mathbb{R}}=\sum_{j=1}^n \mathbb{R}_{\geq 0} N_j
\end{equation}  
be a cone in the lie algebra $\mathfrak{g}$.
Define  
\begin{equation} 
D_{\sigma}:=Spec([\mathbb{C}[\Gamma(\sigma)^{\vee}])_{an} \cong Hom(\Gamma(\sigma)^{\vee},\mathbb{C})
\end{equation} 
where $ \Gamma(\sigma)^{\text{gp}}=\exp(\sigma_{\mathbb{R}}) \cap G_{\mathbb{Z}}, \  \Gamma(\sigma)=\exp(\sigma) \cap G_{\mathbb{Z}} $, with the torus 
\begin{equation} 
T_{\sigma}:=Spec(\mathbb{C}[\Gamma(\sigma)^{\vee \text{gp}}])_{an} \cong Hom(\Gamma(\sigma)^{\vee \text{gp}} , \mathbb{G}_m) \cong \mathbb{G}_m \otimes \Gamma(\sigma)^{\text{gp}} 
\end{equation} 
The superfix '$\text{gp}$' means the group generated, and the suffix '$\text{an}$' means the associated analytic space and we are using the notation from toric geometry. Given $N \in \mathfrak{g}$, the set 
\begin{equation}
\check{B}(N)=\{\ F \in \check{D}\ |\ Ad(e^{\tau N})F \in D, \text{Im}(\tau) \gg 0\ \}
\end{equation} 
denotes all arising possible LMHS associated to $N$. The quotient 
\begin{equation} 
B(N)=e^{\mathbb{C}N} \setminus \check{B}(N)
\end{equation}  
consists of all the LMHS upto reparamentization. We define the boundary component associated to $N$ by 
\begin{equation} 
\overline{B(N)}=\Gamma \setminus B(N)
\end{equation} 
It has been shown in \cite{KP} that there is a tower of fibrations 
\begin{equation} 
B(N) \to ... \to B(N)_k \to ... \to D(N)
\end{equation}  
which passes to the quotient by $\Gamma$ with fibers the intermediate Jacobians, when $\Gamma$ is neat. We give some examples below. Generally the boundary components are explained by the Deligne decomposition of the corresponding LMHS. Thus instead of Hodge diamond we consider the diamond of the $I^{p,q}$ pieces that appear in the Deligne decomposition of the LMHS.

\begin{example} \cite{K}
Consider the degeneration of the middle cohomology of a fibration of projective curves of genus 2, where the HS has weight 2 and is of type $(1,2,1)$. The example can be illustrated geometrically as possible degeneration in a smooth compact curve with two holes, where each hole can degenerate independently or both may degenerate. In the $I^{p,q}$ coordinates we have 4 possible degenerations with $(p,q)$ be
\begin{itemize}
    \item $2(1,0)+2(0,1)$ where $N=0$
    \item $(1,0+(0,1)$ and $(1,1) \to (0,0)$ for $N_1$.
    \item The same as previous item for $N_2$.
    \item $2(1,1) \to 2(0,0)$ corresponding to $N_1+N_2$.
\end{itemize}
where the multiplicities denote the dimension of the $I^{p,q}$ and $N_1$ and $N_2$ are 
\begin{equation}
N_1=\begin{pmatrix} & & 0 & \\ & & & 0\\0 & & & \\ & 1 & & \end{pmatrix}, \qquad  N_2
=\begin{pmatrix} & & 0 & \\ & & & 0\\1 & & & \\ & 0 & & \end{pmatrix}   
\end{equation}
The underlying vector space has dimension 4 and the group $G$ is $Sp_4$. We have the above 4 boundary components. This example is not CY type. 
\end{example}
\begin{example} \cite{K}
Consider the degeneration of a Hodge structure of weight 2 and of type $(1,n,1)$. The possible $I^{p,q}$ components of the degenerating MHS are 
\begin{itemize} 
\item $(2,0)+n(1,1)+(0,2)$
\item $(1,0)+(0,1)+(2,1)+(1,2)+(n-2)(1,1)$
\item $(0,0)+n(1,1)+(2,2)$
\end{itemize}
where we take $n \geq 2$. The nilpotent transformations $N$ are of type $(-1,-1)$ on each components depending to $n$. Each of the above 3 cases appear as different boundary components of the period domain. 
\end{example}

\begin{example} \cite{K}
Consider the example of a Hodge structure of weight 3 and $\dim V=4, \ h=(1,1,1,1)$, where $D=Sp_4(\mathbb{R})/U(1)^{\times 2}$. We shall also illustrate the $(p,q)$-domain of the Deligne decomposition $V=\oplus_{p,q}I^{p,q}$. In this case the only possibilities for $N$ are 
\begin{equation}
    N_1=\begin{pmatrix} 0&0&0&0\\a&0&0&0\\c&b&0&0\\d&c&-a&0\end{pmatrix}, \qquad N_2=\begin{pmatrix} 0&0&0&0\\0&0&0&0\\0&0&0&0\\a&0&0&0\end{pmatrix} \qquad N_3=\begin{pmatrix} 0&0\\A&0&\end{pmatrix}, \ A^t=A>0
\end{equation}
In this case the LMHS associated to $N_1$ i.e. on the boundary $\overline{B(N_1)}\cong \mathbb{C}^*$ has all 4 nonzero $I^{p,q}$ on the diagonal, $(0,0) \leftarrow (1,1) \leftarrow (2,2) \leftarrow (3,3)$ where the action of $N_1$ is as arrows. From the result of \cite{KP} explained above we have $\overline{B(N_2)}\cong \text{non compact elliptic modular surface}$ and the points have nonzero $I^{p,q}$ on $(3,0), (0,3)$ isolated, and $(1,1) \leftarrow (2,2)$. The boundary associated to $N_3$ correspond to two pure Hodge structure $(2,0) \leftarrow (3,1)$ and $(0,2) \leftarrow (1,3)$, and one has $\overline{B(N_2)}$ is isomorphic to a CM elliptic curve.
\end{example}
\begin{example} \cite{KP}
Consider the HS worked out by Carayol, the $\mathbb{Q}$-vector space of dimension $6$ and the Hodge structure of weight 3 of type $(1,2,2,1)$, where $D=U(2,1)/U(1)^3$. The example given by Carayol considers HS with decomposition $V_{\mathbb{Q}(\sqrt{-d})}=V^+ \oplus V^-$, where $\overline{V^+}=V^-$. There are 3 possible boundary components 
\begin{itemize}
    \item $B(N_1) \cong \mathbb{C}^{\times}$ of points with $I^{p,q}$ to be $(3,2) \to (2,1) \to (1,0)$ and $(2,3) \to (1,2) \to (0,1)$.
    \item $B(N_2) \cong \text{CM elliptic curve}$ of points with $I^{p,q}$ to be $(3,1) \to (2,0)+(1,3) \to (0,2)$ and $(2,1) \to (1,2)$.
    \item $B(N_3) \cong \text{complex conjugate of}{\ B(N_2)}$ of points with $I^{p,q}$ to be $(3,0) + (0,3)$ and $2(2,2) \to 2(1,1)$.
\end{itemize} 
\end{example}

\vspace{0.3cm}

\section{Attractor points and degeneration of Hodge structure}

\vspace{0.2cm}

We consider the fibration over $\mathbb{P}^1$ defined by  
\begin{equation} 
y^2=4x^3-G_2(s,t,u)x-G_3(s,t,u)
\end{equation}
It has singular fibres at $t=0, 1, \infty$, such that $s$ and $u$  are affine variables on a surface. The cohomologies of the smooth fibres build up a VHS of type $(1,1,1,1)$, whose periods are given by hypergeometric functions as 
\begin{equation} 
F(s)={}_nF_{n-1}\left (  _{\ \ 1, 1, ...,  1}^{\rho_1,\rho_2, ..., \rho_n} \big | \ s \right )
\end{equation}
In order to write a basis for the vanishing cohomology we may use the form of solutions near $0$. That is the basis 
\begin{equation}
f_j=\frac{1}{(2i\pi)^j j!} \frac{\partial^j}{\partial \alpha^j}  \Big|_{\alpha=0}f_{\alpha}(t) , \qquad f_{\alpha}=t^{\alpha}F(t), \ j=3,2,1,0   
\end{equation}
We form the vector $\varpi=[f_3, f_2, f_1, f_0]^T$ namely period vector. The Hodge filtration on $H^3(X)$ is given by 
\begin{equation}
\begin{aligned}
F^3&=\langle f_3 \rangle\\
F^2&=\langle f_3 , f_2 \rangle\\
F^1&=\langle f_3 , f_2, f_1 \rangle\\
F^0&=\langle f_3 , f_2, f_1, f_0 \rangle
\end{aligned}
\end{equation} 
over the $\mathbb{C}$. Then we need to write down $\varpi$ in a symplectic basis $(\alpha_1, \alpha_2, \beta_1, \beta_2)$. There is a transition function $S$ such that $\Pi=S \varpi$ is the new period vector in the symplectic basis. then the volume form on the fibres $\Omega_t$ can be written as 
\begin{equation}
    \Omega_t= [\alpha_1, \alpha_2, \beta_1, \beta_2].\Pi  
\end{equation}
If $w \in H^3(X, \mathbb{Z})$ then 
\begin{equation}
    w^{3,0}=\overline{w^{0,3}}=C \Omega_t= \overline{C}\overline{\Omega}
\end{equation}
Therefore $w$ defines an attractor if 
\begin{equation}
\begin{aligned}
&w= [\alpha_1, \alpha_2, \beta_1, \beta_2](C \Pi+\overline{C} \overline{\Pi})\\
&(C \Pi+\overline{C} \overline{\Pi}) \in \mathbb{Z}^4 \setminus 0 \end{aligned}
\end{equation}
The matrix $S$ and also the attractor condition has been studied in \cite{Ya}. The solutions at infinity of the Picard-Fuchs equation can be calculated with similar procedure. In that case the period vector is 
\begin{equation} 
\varpi_{\infty}=\begin{bmatrix}F_4 \\ F_3 \\ F_2 \\ F_1 \end{bmatrix}, \qquad 
F_k({\beta},t)=B_kt^{-\beta}{}_nF_{n-1}^{\alpha}\left (  _{\ \ 1+\beta_k-\beta_1,\ ... \hat{1} ...,\  1+\beta_k-\beta_n}^{\ \ \ \ \ \ \ \ \ \  \rho_k,\ \rho_k,\ ...,\ \rho_k} \big | \ \frac{1}{t} \right )
\end{equation} 
as in Section 2. We have to write $\varpi_{\infty}$ in a symplectic basis and $\Pi_{\infty}=S_{\infty}\varpi$. Thus we have analogous equation for the attractor, \cite{DM, KY, Ya1, Ya2, Ya3, Ya4, Ya5}.

Using a Frobenius method argument we may write the periods of the VHS associated to (14) as $\varpi_j(t)=f_j(t)$, that is the same as the periods are given by the hypergeometric functions and their derivatives. It is convenient to locally write these periods as the sum of multivalued sections as 
\begin{equation}
    \varpi_j=\sum_{l=0}^3{3 \choose j}f_j\log(t)^{3-j}
\end{equation}
which have logarithmic singularities and extend the former solutions over the degeneracies. We wish to write the periods in a symplectic basis, denoted by $\Pi$.  By simple linear algebra we can write  
\begin{equation}
    \begin{bmatrix}\Pi_1 \\ \Pi_2 \\ \Pi_3 \\\Pi_4 \end{bmatrix}=\begin{pmatrix} 1 &0&0&0\\ a_{1,0} &1 &0 &0 \\ a_{2,0} &a_{2,1} &1 &0 \\ a_{3,0} &a_{3,1} & a_{3,2} & 1 \end{pmatrix} \begin{bmatrix}\varpi_1 \\ \varpi_2 \\ \varpi_3 \\ \varpi_4 \end{bmatrix}
\end{equation}
In theoretical Physics they study the matrix $A_{4 \times 4}$ in the above identity in examples as $A=A_{\zeta}A_{\log}$, where $A_{\zeta}$ is a matrix with entries in terms of zeta values, and $A_{\log}$ has formal entries in terms of logarithms $a_{\log}(n)=\frac{n \log n}{2\pi \sqrt{-1}}$, \cite{KY}, see also\cite{Ya1, Ya2, Ya3, Ya4, Ya5}.

The understanding of the exchange matrix $A$ above is through mirror symmetry. The variation of HS of the family of CY 3-folds corresponds to the deformation of a germ of Frobenius algebra constructed from the cohomology ring of the fibre varieties. The product structure on these Frobenius algebra maybe explained by a quantum K\"ahler potential on the moduli of K\"ahler forms on the CY 3-fold. In Physics they refer to it as the K\"ahler prepotential on the K\"ahler side of the mirror. The potential has the form 
\begin{equation}
    \phi=-\frac{1}{6}\phi_{111}q^3-\frac{1}{2}\phi_{011}q^2-\frac{1}{2}\phi_{001}q-\frac{1}{6}\phi_{000}+ \phi^{\text{np}}
\end{equation}
where the coefficients $\phi_{ijk}$ are given by certain triple products of cohomology classes in degree 2, and can be calculated in terms of the derivatives of the potential itself. The variable $q$ correspond to the quantum deformation on the potential. The last term $\phi^{\text{np}}$ contains resummed instanton contributions and thus contains all non-perturbative corrections and when $q \to \sqrt{-1}\infty$ it tends to $0$. The mirror map identifies $t$ with a ratio of periods that in the large complex limit shows logarithmic behavior, i.e $t \sim \log q$ which shows if $q \to \sqrt{-1} \infty$ then $t \to \infty$. This leads to the following formula for the matrix $S$, 
\begin{equation}
S=\text{Const} (2\pi \sqrt{-1})^3\begin{pmatrix} -\frac{1}{3}\phi_{000} &-\frac{1}{2}\phi_{001}&0&\frac{1}{6}\phi_{111}\\ -\frac{1}{2}\phi_{111} &-\phi_{001} &-\frac{1}{2}\phi_{111}&0 \\ 1 &0 &0 &0 \\ 0 &1 & 0 & 0 \end{pmatrix}, \qquad \text{Const} \in \mathbb{Q}^{\times}
\end{equation}
The concept of the large complex structure limit comes from the mirror symmetry and the nature of the mirror map between two mirror CY varieties. The point is, under the mirror map the neighbourhood of zero on the complex side correspond to a neighbourhood of $\sqrt{-1}\infty$ in the mirror family (K\"ahler side), \cite{KY, Ya1, Ya2, Ya3, Ya4, Ya5}.  

The attractor condition may be studied as  a degenerating situation for the cohomology class $w$ to collapse through a family to a class in $H^{3,0} \oplus H^{0,3}$. That is we consider a cohomology class $\mathfrak{w}$ on a thickening of $X$ over a suitable quasiprojective base as $\mathfrak{X} \to S^*$ which restricts to a $w_s$ on the generic fibre $\mathfrak{X}_s \cong X$, and decompose it regarding the Hodge decomposition. The phenomenon can be studied in a universal way over the moduli $M$. The limit definition in a one parameter family of Hodge flags defined by a nilpotent transformation $N \in \text{End}(V)$ can be expressed as
\begin{equation}
\lim_{\text{Im}(z) \to \infty}\exp(-\sqrt{-1}z.N)\left ( \bigoplus_{p+q=3}H^{p,q}(X_z) \right ) \subset H^{3,0} \oplus H^{0,3}
\end{equation}
Because $N^r=0$ for some $r \leq 3$ the exponential series involving $N$ actually has at most 4 terms. having $N$ to be specific while the vectors in the Hodge decomposition are written in terms of periods, the above limit gives equations defining constrains for attractor points. One should note that attractor point must also lie in the integral lattice $H^3(X, \mathbb{Z})$.

The above procedure depends to the choice of the nilpotent transformation $N$. In fact according to the explanation in Section 4 it depends to the nilpotent cone $\sigma$ defining the corresponding boundary component. Therefore one has to consider the associated constrains for the attractor locus on different boundary components separately. We do this in an example below. 

We calculate the aforementioned constrains for the Example 4.3 of the previous section. \begin{itemize} 
\item \textbf{Attractor constrains on the boundary component $\overline{B(N_3)}$}: in this case 
\begin{equation}
\exp(-\sqrt{-1}z.N_3)=I+\sqrt{-1}z.N_3, \qquad N_3^2=0
\end{equation}
It follows that at an attractor point we must have 
\begin{equation}
\lim_{\text{Im}(z) \to \infty} \left ( I+\begin{pmatrix} 0&0\\\sqrt{-1}z.A&0&\end{pmatrix}\right ) \begin{bmatrix} \Pi_1 \\ \Pi_2 \\ \Pi_3 \\ \Pi_4 \end{bmatrix} \ \text{is of type (3,0)+(0,3)}   
\end{equation}
where $A_{2 \times 2}$ is symmetric. Setting $A=\begin{bmatrix} a & b \\ b & d \end{bmatrix}$ we obtain 
\begin{equation}
    \begin{bmatrix} \Pi_1 \\ \Pi_2 \\ \sqrt{-1}az\Pi_1+\sqrt{-1}bz\Pi_2+\Pi_3 \\ \sqrt{-1}bz\Pi_1+\sqrt{-1}dz\Pi_2+\Pi_4 \end{bmatrix} \ \text{is of type (3,0)+(0,3)}
\end{equation}
Therefore we have 
\begin{equation}
\begin{aligned}
&\Pi_2=0 \\
&\lim_{z \to \infty}az\Pi_1+bz\Pi_2+\Pi_3=0\\
&\overline{\Pi_1}=\lim_{z \to \infty}\sqrt{-1}bz\Pi_1+\sqrt{-1}dz\Pi_2+\Pi_4
\end{aligned}
\end{equation}
Here $\Pi$ is the symplectic period vector. \item \textbf{Attractor constrains on the boundary component $\overline{B(N_2)}$}: with the same method we obtain the following equations
\begin{equation}
\begin{aligned}
&\Pi_2=\Pi_3=0 \\
&\overline{\Pi_1}=\lim_{z \to \infty} az\Pi_1+\Pi_3
\end{aligned}
\end{equation}
\item \textbf{Attractor constrains on the boundary component $\overline{B(N_1)}$}: in this case $N_1^4=0$ and $\exp(-\sqrt{-1}z.N_1)=I-\sqrt{-1}z.N_1+z^2N_1^2-\sqrt{-1}z^3N_1^3$ and the condition is
\begin{equation} 
\lim_{z \to \infty} \left ( I-\sqrt{-1}z.N_1+z^2N_1^2-\sqrt{-1}z^3N_1^3 \right ) \begin{bmatrix} \Pi_1 \\ \Pi_2 \\ \Pi_3 \\ \Pi_4 \end{bmatrix} \ \text{is of type (3,0)+(0,3)}
\end{equation} 
Replacing the matrix of $N_1$ from the previous section one can simply form the constrains of attractors that may exist on this component.
\end{itemize}
One has to note that the attractor points are the points in the integral cohomology of $X$ that lie in the locus of these constrains. There may also exists other boundary components associated to the choice of different positive nilpotent cones in the  
Lie algebra of the symmetries of $H^3(X, \mathbb{C})$.

From the nilpotent orbit theorem of W. Schmid near the boundary component associated to $N_i$ we have the vector space identity
\begin{equation}
F^p(t)=e^{t_iN_i}e^{\Gamma(z)}F_0^p
\end{equation}
where $\Gamma(z)$ is a $\mathfrak{sp}(2h^{2,1}+2)$-valued map with $\Gamma(0)=0$ and $F_0^p$ is the limiting Hodge filtration and does not depend to the coordinates $t_i$. The nilpotent orbit theorem further says that the Hodge filtration $F^p$ can be approximated by its nilpotent orbit $F_{\text{nil}}:=e^{t_iN_i}F_0^p$.
The task here is to translate the asymptotic behavior of the Hodge filtration $F^p$ into equivalent statement about the period vector $\Pi$. We can write the period vector 
\begin{equation}
\Pi(t)= e^{t_iN_i}e^{\Gamma(z)}w_0=e^{t_iN_i}(a_0+a_{1,i_1}e^{2 \pi \sqrt{-1}t_{i_r}}+ \dots )
\end{equation}
where we have expanded $\Gamma(z)$ in the variable $z=e^{2\pi \sqrt{-1}t}$. The first term could be thought as a polynomial approximation of the period. The terms starting from the second are instanton terms called non-perturbative corrections. 

In 1-dimensional moduli space $h^{2,1}=1$ case we have 3 kinds of boundaries; the first type of boundary is the conifold point; in this case $\varpi$ is given by 
\begin{equation}
\varpi=\begin{bmatrix} 1+\frac{a^2}{8 \pi}z^2\\ az \\ \sqrt{-1}-\frac{\sqrt{-1}a^2}{8 \pi}z^2\\
 \frac{\sqrt{-1}a}{2 \pi}z\log z\end{bmatrix}
\end{equation}
The second type of boundary limit is the Tyurin degeneration; in this case the period vector can be written as 
\begin{equation}
\varpi=\begin{bmatrix} 1+az\\ \sqrt{-1}-\sqrt{-1}az \\ \frac{\log z}{2 \pi \sqrt{-1}}-\frac{az}{2 \pi \sqrt{-1}}(\log z-2)\\
 \frac{\log z}{2 \pi}-\frac{az}{2 \pi}(\log z-2)
\end{bmatrix}
\end{equation}
The thirs type is large complex limit structure point. This case was discussed before, see also \cite{Ya1, Ya2, Ya3, Ya4, Ya5} for details and examples. In higher dimensional moduli, instead of considering the limit where all the coordinates are sent to the limit at the same rate, we can take ordered limit, such that first $y_1 \to \infty$ and then $y_2 \to \infty$ etc. where $y_1 \gg y_2$ and so on. We obtain a chain of boundary components 
\begin{equation}
I \stackrel{y_1 \to \infty}{\longrightarrow} Type A_{(1)} \stackrel{y_2 \to \infty}{\longrightarrow} Type A_{(2)} \to \dots \stackrel{y_n \to \infty}{\longrightarrow} Type A_{(n)}
\end{equation}
where we have written the variable $t_i=x+\sqrt{-1}y_i$, cf. \cite{BGH01}. 

\section{Relation with GIT and moment map}

Assume we have a symplectic manifold $(M,\omega, J)$ that is $M$ is closed compact K\"ahler manifold with symplectic form $\omega$ and $J$ an integrable complex structure such that $\omega(.,.)$ is a Riemannian metric. Also assume we have an action of a Lie group ${G}_\mathbb{R} \subset U(n)$ on $M$ which preserves all the 3 structures $\omega, J$ and the metric. The Lie algebra $\mathfrak{g}=Lie (G) \subset \mathfrak{u}(n)$ has an infinitesimal action on the vector fields over $M$ written as 
\begin{equation}
\xi \mapsto [X \mapsto X_{\xi}]
\end{equation} 
Equip $\mathfrak{g}$ with an invariant bilinear form such that the $G$-action gets Hamiltonian. If 
\begin{equation} 
\mu:M \to \mathfrak{g}
\end{equation} 
is a moment map for the action. Assume  $X_{\xi}$ is a Hamiltonian vector field such that 
\begin{equation}
L_{X_{\xi}}\omega =dH_{\xi}, \qquad H_{\xi}(.)=\langle \mu(.),\xi \rangle
\end{equation} 
We can write this as 
\begin{equation} 
\langle d\mu(x) v,\xi \rangle =\omega({X_{\xi}}(x),v), \qquad v \in T_xM
\end{equation} 
Let $G_{\mathbb{C}}$ be the complexification of $G$ and $\mathfrak{g}_{\mathbb{C}}=\mathfrak{g}+\sqrt{-1}\mathfrak{g}$ its Lie algebra. Taking the group of holomorphic automorphisms of $X$ we obtain an action of $G_{\mathbb{C}}$ on $X$ and the infinitesimal action gets the form 
\begin{equation}
\xi=\alpha+\sqrt{-1}\beta \longmapsto v_{\xi}=v_{\alpha}+Jv_{\beta}
\end{equation}
where $v_{\alpha}$ is the Hamiltonian vector field of the function $H_{xi}$ and $Jv_{\beta}=\nabla H_{\beta}$ is the gradient vector field of the function $H_{\beta}$. A well known fact is that; two points in the zero set of the moment map are equivalent under $G_{\mathbb{C}}$ if and only if they are equivalent under $G$, see \cite{VRS}. One may express this 
\begin{equation}
M//G \cong M^{ps}/G_{\mathbb{C}}
\end{equation}
where $M//G=\mu^{-1}(0)/G$ and 
\begin{equation}
M^{ps}=\{x \in M| G_{\mathbb{C}}.x \cap \mu^{-1}(0) \ne \emptyset\}
\end{equation} 
These spaces are generally singular. A point $x \in \mu^{-1}(0)$ is regular if and only if its isotropy subgroup $G_x$ is discrete. When $\mu(x)=0$ then the isotropy subgroup $G_{\mathbb{C},x}$ is the complexification of $G_x$. Let $M^s \subset M^{ps}$ be the open subset of points with discrete isotropy $G_{\mathbb{C},x}$. It follows that $M^s//G \cong M^s/G_{\mathbb{C}}$. In general $M^{ps}$ is not an open subset of $M$. In geometric invariant theory one studies $M^{ss}/G_{\mathbb{C}}$ where 
\begin{equation}
M^{ss}=\{x \in M|\overline{G_{\mathbb{C}}.x} \cap \mu^{-1}(0) \ne 0\}.
\end{equation}
see \cite{VRS} for details. 

We consider the function defined by the square of the norm of the moment map, i.e the function 
\begin{equation}
\begin{aligned}
f:& \ 
X \longrightarrow \mathbb{R}\\
f(x)&=\frac{1}{2}|\mu(x)|^2
\end{aligned}
\end{equation}
The gradient of $f$ is given by $\nabla f(x)=J v_{\mu(x)}(x)$. It is convenient to denote $L_x \xi=v_{\xi}(x)$. The integral curves of the differential equation 
\begin{equation} 
(*): \dot{x}=- \nabla f(x), \qquad x(0)=x_0
\end{equation} 
play a crucial role in GIT. We first state the following.
\begin{theorem} [Convergence Theorem] \cite{VRS}
Let $x(t)$ be an integral curve of the equation (*) passing from an initial point $x_0 \in M$. Then the limit $x_{
\infty}=\lim_{t \to \infty}x(t)$ exists. 
\end{theorem}
The existence of the limit in the above theorem is characteristic and goes back to the fact that the quadratic function $f$ is a Morse function. The next proposition asserts that the function $f$ attains its minima along the $G_{\mathbb{C}}$-orbit at the critical points.
\begin{proposition} [Kirwan-Ness Inequality] \cite{VRS}
Let $x \in M$ be a critical point of the moment map squared. Then
\begin{equation}
|\mu(x)| \leq |\mu(gx)|, \qquad \forall g \in G_{\mathbb{C}}
\end{equation} 
\end{proposition}
If the set of critical points of $f$ is nonempty then it consists of the absolute minima of $f$ and every negative gradient flow converges to a critical point. 
\begin{theorem} [Moment Limit Theorem] \cite{VRS}
Let $x:\mathbb{R} \to M$ be a solution of (*), and set $x_{\infty}=\lim_{t \to \infty}x(t)$. Then 
\begin{equation}
|\mu(x_{\infty})|=\text{inf}_{g \in G_{\mathbb{C}}}|\mu(gx_0)|
\end{equation}
Moreover, the $G$-orbit of $x_{\infty}$ depends to the $G_{\mathbb{C}}$-orbit of $x_0$. 
\end{theorem}
Thus the limits of integral curves of the (*) tend to the minimum of $|\mu(x)|$, \cite{VRS}. We have the following more general theorem.
\begin{theorem} [Moment-Weight Inequality] \cite{VRS}
Define the $\mu$-weight of the pair $(x, \zeta) \in M \times [\mathfrak{g} \smallsetminus 0])$ by 
\begin{equation}
w_{\mu}(x, \zeta):=\lim_{t \to \infty}\langle \mu(\exp(\sqrt{-1}t\zeta)x,\zeta \rangle
\end{equation}
Then this limit exists and defines a $G_{\mathbb{C}}$-invariant function $w_{\mu}$. Moreover, for every $x \in M$ and $\xi \in \mathfrak{g} \smallsetminus 0$ we have 
\begin{equation}
\frac{|w(x,\xi)|}{|\xi|} \leq |\mu(gx)|, \qquad g \in G_{\mathbb{C}}
\end{equation}
If $x_0 \in M$ is such that $\overline{G_{\mathbb{C}}.x_0} \cap \mu^{-1}(0)=\emptyset$ then there exists $\xi_0 \in \mathfrak{g}$ with $|\xi_0|=1$ and 
\begin{equation}
-|w(x_0,\xi_0)|= \text{inf}_g|\mu(gx_0)| =sup_{\xi}\frac{|w(x_0,\xi)|}{|\xi|}, \qquad g \in G_{\mathbb{C}}, \xi \in \mathfrak{g} \smallsetminus 0
\end{equation} 
\end{theorem}
Consider the case $M=\mathbb{P}(V)$ or when $M$ is a $G$-invariant quasiprojective subvariety of $\mathbb{P}(V)$. A candidate for the moment map of the action of $G$ is defined by 
\begin{equation}
\langle \mu(x), \xi \rangle =a \times \frac{\langle v, \sqrt{-1}\xi v \rangle}{|v|^2}, \qquad x=[v] \in \mathbb{P}(V), \ a \in \mathbb{R}^+
\end{equation}
where $a>0$ maybe chosen arbitrary. This follows from the definition of the moment map at the beginning of this section. In this case the moment-weight function $w(x, \xi)$ has simple explanations in terms of the eigenvalues of the matrix of the operator $\sqrt{-1}\xi$, \cite{VRS}. The functions $f$ in (*) and $w(x, \xi)$ may be interpreted as certain quadratic functions defined by a bilinear pairing, which have Morse property. Their critical values naturally appear to have the maximum or minimum properties as the reader may expect from differential geometry.

Our purpose is to provide another interpretation of the attractor locus in terms of the critical locus of the absolute value of a moment map squared. We come back to the functions $Z_{\gamma}$ defined by (49). Recall also that the attractor points can was explained by the critical locus of the gradient flow of $|Z_{\gamma}|$. We shall consider the function 
\begin{equation}
Z_{\gamma}=\frac{Q^t \Sigma \Pi}{(\Pi^{\dagger}\Sigma \Pi)^{1/2}} = \langle \mu(Q), \Sigma \rangle 
\end{equation}
as the equation defining the moment map. We have the following result.

\begin{theorem} [Main Result]
Assume the moment map $\mu:M \to \mathfrak{g}=Lie(Aut(\Sigma)_{\mathbb{C}})$ is defined by the equation (110). The attractor points on the moduli of complex structure $M$, are the critical points of the negative gradient flows of $-\nabla |Z_{\gamma}|$. The absolute value of the momentum at the attractor points can be defined by the moment-weight function $w(Q_0,\Sigma)$ for an unstable point $Q_0$. 
\end{theorem}
\begin{proof}
By Theorem 7.2 the critical points $Q_0^t$ of the flow of $-\nabla |Z_{\gamma}|$ is characterized by the inequality $|\mu(Q_0)| \leq |\mu(g.Q_0)|$ where $g$ runs through the symmetry group of $\Sigma$ as a bilinear form, i.e. $G=Aut(\Sigma)_{\mathbb{C}}$. Theorem 7.3 ensures that the critical points are the limit points on the integral curves of $-\nabla |Z_{\gamma}|$. We can define the function 
\begin{equation} 
w(x,\Sigma)=\lim_t \langle \mu(\exp (\sqrt{-1}t\Sigma)x, \xi \rangle
\end{equation} 
Theorem 7.4 applies here and we obtain the result. 
\end{proof}
The above theorem gives an alternative description for attractor points on the moduli space $M$. In this way some of the properties of these points maybe understood from the GIT analysis of the moment maps. 
The moment map plays a crucial role in the study of semistable (resp. stable) locus in  symplectic geometry. 
\vspace{0.3cm}

\section{Relation with CY modularity conjecture}

\vspace{0.3cm}

The section is expository to express the relation of attractor CY varieties to the famous modularity conjecture. Some relevant references are \cite{COFKM, CO, COES, CGP, Del, DK, DSh, KY, KNY, M, Mi, S, Schu, Si1, Si2, ShI, Ya1, Ya2, Ya3, Ya4, Ya4, Ya5, Yu2, Yu3}. If $X$ is defined over $\mathbb{F}_p$, then the Artin Zeta function of $X$ is defined by
\begin{equation}
Z(X/\mathbb{F}_p,t)=\exp \left (\sum_n \frac{| X(\mathbb{F}_{p^n})|}{n}t^n \right )
\end{equation}
where $| X(\mathbb{F}_{p^n})|$ denotes the number of the points in the reduction of the algebraic variety over $\mathbb{F}_{p^n}$. Dwork conjectured that the generating function $Z$ is rational. Later Deligne proved this conjecture in the framework of Weil conjectures. first note that the Lefschetz fixed point theorem states that if $f:X \to X$ is a continuous endomorphism of quasi-projective variety $X$, then 
\begin{equation}
\Gamma_f.\Delta=\sum (-1)^iTr(f^*|_{H^i(X,\mathbb{Q}_l)})
\end{equation}
where $\Gamma$ is the diagonal in $X \times X$. It follows that $\Gamma_{\text{Frob}^r}.\Delta=|X(\mathbb{F}_{q^r})|$. It follows that 
\begin{equation}
\det(1-\text{Frob}_{\mathfrak{p}}|_{H^*(\overline{X_{et}}, \mathbb{Q}_l)})=\exp(\sum_{r=0}^{\infty}|X(\mathbb{F}_{q^r})|\frac{t^r}{r} )
\end{equation}
Assume $X$ is a non-singular projective variety defined over the number field $K$, define the $L$-series of $X$ by 

\begin{equation}
L(X,s)=\prod_i \prod_{\mathfrak{p} \ne \infty} \det(1-\text{Frob}_{\mathfrak{p}}N\mathfrak{p}^{-s}|_{H_{et}^i(\overline{X},\mathbb{Q}_p)^{I_{\mathfrak{p}}}})^{{(-1)}^{i+1}}
\end{equation}
where $N\mathfrak{p}=|\kappa(\mathfrak{p})|$ is the order of the residue field of the $K$ at $\mathfrak{p}$. The modularity conjecture asserts that, $L(X,s)=L(f,s)$ for a Hecke eigenform $f$. This means that if $L(X,s)=\sum a_nn^{-s}$ and $f=\sum c_nq^n$ is the $q$-expansion of the modular form $f$, then $a_n=c_n$. In other words $L(X,s)$ is the $L$-function of the modular form $f$.

\begin{theorem} (Weil conjectures-P. Deligne)
\begin{itemize}
\item[(1)] $Z(X,t)$ is a rational function of $t$ and can be written 

\begin{equation}
Z(X,t)=\frac{P_1(t)...P_{2n-1}(t)}{P_0(t)...P_{2n}(t)}, \qquad n=\dim X, \ P_i(0)=1
\end{equation}

\noindent
where $P_i(t)$ is a polynomial of degree $\beta_i$ the $i$-th Betti number of $X$. Moreover $P_0(t)=1-t$ and $P_{2n}(t)=1-q^n.t$. 
\item[(2)] One has a functional equation

\begin{equation}
Z(X,\frac{1}{q^n.t})=\pm q^{n \chi /2} \ t^{\chi}Z(X,t), \qquad \chi=\sum_{i=0}^{2n}(-1)^i\beta_i
\end{equation}
\item[(3)] The polynomials $P_i(t)$ above have integral coefficients and if 

\begin{equation}
P_i(t)=\prod_i(1-w_{\alpha}.t)
\end{equation}

\noindent
the complex numbers $w_{\alpha}$ have absolute value $q^{1/2}$.
\end{itemize}
\end{theorem}
Assume $X$ is a CY 3-fold, the most interesting cohomology of $X$ is the Galois module $H^3(\overline{X}, \mathbb{Q}_l)$. For this we may consider the $L_p$-factor  
\begin{equation} 
L_p(X,s)=(*)\prod_{p \ \text{good}}P_3(X,p^{-s})^{-1}
\end{equation}
where a good prime means the primes that the reduction of $X$ at $p$ is nonsingular and $(*)$ is a finite product of Euler factors at the bad primes. The attractor formalism suggests to look for a decomposition of the motive of $H^3(X, \mathbb{Q}_l)$ in the form $M_{\text{flux}} \oplus M_{\text{rest}}$ where $M_{\text{flux}}$ is of Hodge type $(2,1)+(1,2)$ and $M_{\text{rest}}$ is of Hodge type $(3,0)+(0,3)$. When looking at the mirror CY we have the total Hodge type of $H^3$ of the form $(1,1,1,1)$, i.e. a $\mathbb{Q}$-vector space of dimension $4$. It follows that one may consider the above two pieces as a motive of an elliptic curve, where the modularity conjecture has a solution. In this case the characteristic polynomial of the $Frob_p$ on $H^3$ which is a polynomial of degree $4$, is decomposed as the product of two polynomials of degree $2$. Therefore a premier step in checking the modularity conjecture is to look for primes $p$ where the characteristic polynomial $P_3(X,t)$ can be written as the product $P_3^{(1)}P_3^{(2)}$ of two polynomials of degree $2$. On the other hand one can read the coefficient of the potential modular form satisfying the modularity question from the polynomials $P_3^{(1)}$ and $P_3^{(2)}$. It follows that with a bit of chance one may be able to find out a positive answer to the modular CY 3-fold from the table of known modular forms in a specific example, see \cite{COES}.

\begin{remark} 
We shall be dealing with projective CY hypersurface $X$ in $\mathbb{P}(k_1, \dots, k_5)$ given by equations of the form
\begin{equation}
\sum_{i=1}^5 x_i^{K/k_i}-\phi_Ax^A=0
\end{equation}
where $x^A$ are monomials of total weight $K$ and $\phi_A$ parameterize $h^{2,1}$-dim complex structure moduli. The Hasse zeta function of $X$ can be written as $\zeta(X,t)=R(t)/D(t)$, where we are mainly interested to the numerator $R(t)$, which is a polynomial of degree $2+h^{2,1}(X)$. The mirror $Y$ of $X$ has an equation in the form 
\begin{equation}
\sum_{i=1}^5 y_i^{\tilde{K}/\tilde{k}_i}-\phi_{\tilde{A}}x^{\tilde{A}}=0
\end{equation}
according to a construction by Greene-Plesser. We may write the $\zeta$ function of $Y$ as $\zeta(Y,t)=R'(t)/D'(t)$ and $R'(t)$ is a polynomial of degree $2+h^{2,1}(Y)=2+h^{1,1}(X)$. According to the Greene-Plesser construction from the beginning we may write the equation of $X$ in the form 
\begin{equation}
\sum_{i=1}^5 x_i^{K/k_i}-\psi_{\bar{A}}x^{\bar{A}}-\phi_Ax^A=0
\end{equation}
where $\bar{A}$ runs over $h^{1,1}(X)$ parameters and $A$ over $h^{2,1}-h^{1,1}$. In this case the observation is that $R(t)=R'(t)R_1(t)$, see the Appendix in \cite{COES}. 
\end{remark}

\begin{example} \cite{COES, M}
The working example studied in \cite{COES} is the following octic 
\begin{equation} \label{ex:octic}
x_1^8+x_2^8+x_3^4+x_4^4+x_5^4-8\psi  x_1x_2x_3x_4x_5-2\phi x_1^4x_2^4=0
\end{equation}
in $\mathbb{P}(1,1,2,2,2)$. It is known that this model particularly admits supersymmetric flux compactification. Its Hasse zeta function has been computed for small primes. The generic fibre of the this family of CY-varieties has Hodge numbers $h^{1,1}=2, h^{2,1}=86$. The generic fibre in the mirror family has Hodge
structure of type $(1,1,1,1)$. We are going to check out the modularity conjecture of CY 3-folds for specific fibres and specially we are interested to the factorization of of zeta function. The first interesting CY variety is the fibre at $\psi=\phi=0$. One can check that for primes $p=3,5,7,11, 13, 17$ 
the characteristic polynomial of the $Frob_p$ acting on $H_{et}^3$, mentioned by  $P_3(X,t)$ factors as the product of two polynomials of degree 2. The corresponding coefficients of the $L$-function are as follows 
\begin{equation}
c_3=0, c_5=2, c_7=0, c_{11}=0, c_{13}=-6, c_{17}=2
\end{equation}
Using the table of modular functions one finds that there is a unique $wt=2$ Hecke eigenform with these coefficients. It is easy to check that this modular form will fill out the modularity conjecture for the CY variety at $\phi=\psi=0$ of the family  \eqref{ex:octic}. Lets add that the degeneration of Hodge structure at $(0,0)$ depends to the way of tending to this point in the 2-plane. The fibre at $\psi=0, \phi=1/97$ is another modular CY 3-fold. The next modular fibres along the line $\psi=0$ appear at $\phi=1/2, 3/5, 2$. The details can be found in \cite{COES}.
\end{example}
\noindent 
The above kind of adjustment depends to a version of Hodge conjecture and the Grothendieck conjecture D. The Grothendieck conjecture D for $X$ asserts that if an algebraic cycle $\gamma \in CH^d(X)$ is numerically equivalent to $0$ then its Poincare dual $cl(\gamma)$ is $0$ in the cohomology group $H^{2d}(X)$. The point is on the existence of the decomposition of the motive of the variety $X$ which we denote by $M$ as $M=M_{flux} \bigoplus M_{rest}$. The aforementioned decomposition on the realization $H^3(X)$ could be lifted to the motive itself if and only if the Hodge conjecture and the Grothendieck conjecture D hold. This assumption imply that the realization functor factorizes through $M_{num}$ [the category of motive under numerical equivalence]. The method of checking with evidences above for modularity has been applied to several examples in \cite{CGP, COES, KNY, M} and results  attractor CY 3-folds which are modular. 
There is one more point that we have to mention that, the decomposition of the motive $M=M_{flux} \bigoplus M_{rest}$ usually holds over a finite extension of $\mathbb{Q}$. This fact also applies to the attractor formalism. We state this in the following conjecture.

\begin{conjecture} (Attractor conjecture) \cite{M} \label{thm:AttConj}
Assume $X$ is a polarized CY 3-fold and $\gamma \in H_3(X, \mathbb{Z})$ defines an attractor point in the moduli of complex structure $M$ for $X$. Then
\begin{itemize}
\item the period vector has coordinates  valued in a number field $K(\gamma)$. In other words regarding as a point in a projective space it lies in $\mathbb{P}^{h^{2,1}}(K(\gamma))$.   
\item the corresponding variety $X_{\gamma}$ is arithmetic, i.e. is defined over a number field $K'(\gamma)$. In other words there is an embedding $M \hookrightarrow \mathbb{P}^N$ whose image is defined over the number field $K'(\gamma)$. 
\item the aforementioned embedding maybe chosen so that $K'(\gamma)K(\gamma)$ is Galois over $K$.
\end{itemize}
\end{conjecture} 
The geometry of the attractor CY varieties is very much tied with their arithmetic and class field theory. For instance; for CY 3-folds Elliptically fibred over a K3 surface we we have the following. First we write $H^3(K3 \times E, \mathbb{Z})=H^2(K3, \mathbb{Z}) \bigoplus H^2(K3, \mathbb{Z})$. Thus we may take $\gamma =p \oplus q$ for $p, q \in H^2(K3, \mathbb{Z})$: In fact to do this one needs to solve the equations 
\begin{equation}
\begin{aligned}
2 \text{Im} \overline{C} \int_{a \times \gamma^I}dz \wedge \Omega^{2,0}=p\\
2 \text{Im} \overline{C} \int_{b \times \gamma_I}dz \wedge \Omega^{2,0}=q
\end{aligned}
\end{equation}
where $(\gamma^I , \gamma_I)$ is a symplectic basis for $H_2(K3 , \mathbb{Z})$ and $(a,b)$ are the same for $E$. We find that $\Omega^{2,0}=q-\bar{\tau}p$ where $\tau$ satisfies the equation
\begin{equation}
\int \Omega^{2,0} \wedge \Omega^{2,0} =0 \Rightarrow \langle p,p \rangle \tau^2 -2 \langle p,q \rangle \tau +\langle q,q \rangle =0
\end{equation}
Let $D$ be the discriminant of this equation. The following theorem is known.
\begin{theorem}
There is a 1-1 correspondence between attractor K3 surfaces and $PSL_2\mathbb{Z}$-equivalence classes of positive definite binary quadratic forms. Suppose $(p,q) \in H^2(K3, \mathbb{Z})$ defines a primitive lattice. Then the attractor variety determined by $\gamma=(p,q)$ is $E_{\tau} \times Y_{2Q}$ where 
\begin{equation} 
\tau=\frac{\langle p,q \rangle+\sqrt{-D}}{\langle p, p \rangle}
\end{equation} 

and $Y_{2Q}$ is the Shioda-Inose K3 surface associated to the quadratic form 
\begin{equation}
2Q=\begin{pmatrix}\langle p,p \rangle & -\langle p, q \rangle \\ -\langle p, q \rangle &\langle q, q \rangle \end{pmatrix}
\end{equation}

\end{theorem} 
The absolute value of the associated numerical function $Z_{\gamma}$ at the attractor point is calculated as $|Z_{\gamma}(X_{\gamma}, \gamma)|=\sqrt{-D}$ where $D \leq 0$. In this case the number field under argument of the conjecture \ref{thm:AttConj} can be explicitly calculated in terms of the quadratic extension $\mathbb{Q}[\sqrt{-D}]$ and special values of the $j$-function, see \cite{M}.

\end{document}